\documentclass[12pt]{amsart}
\usepackage{amsmath,amssymb,fullpage,caption,amsthm,graphicx}
\usepackage{tikz}
\usetikzlibrary{calc}
\usetikzlibrary{shapes.symbols}
\usetikzlibrary{arrows}
\usetikzlibrary{decorations.markings,decorations.pathmorphing}
\usetikzlibrary{shapes}
\newcommand{\size}[1]{\left \vert #1 \right \vert}

\newcommand{\cnum}{c} 
\newcommand{\cnumprot}{c} 
\newcommand{\cnumdir}{c} 
\newcommand{\capt}{\mathrm{capt}}
\newcommand{\captprot}{\mathrm{capt}}
\newcommand{\captdir}{\mathrm{capt}}

\newtheorem{theorem}{Theorem}[section]
\newtheorem{lemma}[theorem]{Lemma}

\newtheorem{prop}[theorem]{Proposition}

\newtheorem{cor}[theorem]{Corollary}

\begin{document}

\title{Bounds on the length of a game of Cops and Robbers}

\author{William B. Kinnersley}
\address{Department of Mathematics, University of Rhode Island, University of Rhode Island, Kingston, RI,USA, 02881}
\email{\tt billk@uri.edu}


\begin{abstract}
In the game of Cops and Robbers, a team of cops attempts to capture a robber on a graph $G$.  All players occupy vertices of $G$.  The game operates in {\em rounds}; in each round the cops move to neighboring vertices, after which the robber does the same.  The minimum number of cops needed to guarantee capture of a robber on $G$ is the {\em cop number} of $G$, denoted $\cnum(G)$, and the minimum number of rounds needed for them to do so is the {\em capture time}.  It has long been known that the capture time of an $n$-vertex graph with cop number $k$ is $O(n^{k+1})$.  More recently, Bonato, Golovach, Hahn, and Kratochv\'{i}l (\cite{BGHK09}, 2009) and Gaven\v{c}iak (\cite{Gav10}, 2010) showed that for $k = 1$, this upper bound is not asymptotically tight: for graphs with cop number 1, the cop can always win within $n-4$ rounds.  In this paper, we show that the upper bound is tight when $k \ge 2$: for fixed $k \ge 2$, we construct arbitrarily large graphs $G$ having capture time at least $\left (\frac{\size{V(G)}}{40k^4}\right )^{k+1}$.  

In the process of proving our main result, we establish results that may be of independent interest.  In particular, we show that the problem of deciding whether $k$ cops can capture a robber on a directed graph is polynomial-time equivalent to deciding whether $k$ cops can capture a robber on an undirected graph.  As a corollary of this fact, we obtain a relatively short proof of a major conjecture of Goldstein and Reingold (\cite{GR95}, 1995), which was recently proved through other means (\cite{Kin15}, 2015).  We also show that $n$-vertex strongly-connected directed graphs with cop number 1 can have capture time $\Omega(n^2)$, thereby showing that the result of Bonato et al.~\cite{BGHK09} does not extend to the directed setting.  
\end{abstract}

\maketitle

\begin{section}{Introduction}

In the classic game of {\em Cops and Robbers}, $k$ cops attempt to capture a single robber on a graph $G$.  The cops and robber all occupy vertices of $G$.  (Multiple entities may occupy a single vertex simultaneously.)  Cops and Robbers is a perfect-information game; in particular, the cops and robber know each other's positions at all times.  At the beginning of the game, the cops choose their starting positions, after which the robber chooses his position in response.  Henceforth the game proceeds in {\em rounds}, each of which consists of a {\em cop turn} followed by a {\em robber turn}.  In the usual formulation of the game, on the cops' turn, each cop may either remain on her current vertex or move to an adjacent vertex; on the robber's turn, the robber may do the same.  If at any point any cop occupies the same vertex as the robber, we say that she {\em captures} the robber, and the cops win.  Conversely, the robber wins by perpetually evading capture.

Clearly, placing one cop on each vertex of $G$ results in a win for the cops.  It is thus natural to ask for the minimum number of cops needed to guarantee capture of a robber on $G$, no matter how skillfully the robber plays; this quantity is the {\em cop number} of $G$, denoted $\cnum(G)$.  Much attention has been given to the cop number.  Quilliot~\cite{Qui78}, along with Nowakowski and Winkler~\cite{NW83}, independently introduced the game and characterized graphs with cop number 1, while Aigner and Fromme~\cite{AF84} were the first to study graphs with larger cop numbers.  Since then, there have been numerous papers written on Cops and Robbers, and many variants of the game have been introduced; for more background on the game, we direct the reader to~\cite{BN11}.

In this paper, rather than asking how many cops are needed to capture a robber on $G$, we instead ask: how quickly can they do so?  When $G$ is a graph with cop number $k$, the {\em capture time} of $G$ is the minimum number of rounds it takes for $k$ cops to capture a robber on $G$, provided that the robber evades capture as long as possible.  To be more precise, the capture time of $G$ is the minimum number $t$ such that $\cnum(G)$ cops can always capture a robber on $G$ within $t$ rounds, regardless of how the robber plays.  (We do not consider the players' choosing of initial positions to constitute a round; rather, the first round begins with the first cop turn.)  We denote the capture time of $G$ by $\capt(G)$.

The concept of capture time was first introduced by Bonato, Golovach, Hahn, and Kratochv\'{i}l in \cite{BGHK09}.  They and Gaven\v{c}iak \cite{Gav10} showed that for $n \ge 7$, if $G$ is an $n$-vertex graph with cop number 1, then $\capt(G) \le n-4$ (and this bound is best possible).  For graphs with higher cop number, not much is known about the capture time.  The capture time of two-dimensional Cartesian grids was determined by Mehrabian in~\cite{Meh11}, and the capture time of the $n$-dimensional hypercube $Q_d$ was shown to be $\Theta(d \log d)$ by Bonato, Gordinowicz, Kinnersley, and Pra\l{}at in~\cite{BGKP13}.  More recently, Bonato, P\'{e}rez-Gim\'{e}nez, Pra\l{}at, and Reiniger~\cite{BPPR17} investigated the length of the game on $Q_d$ when playing with more than $c(Q_d)$ cops and similarly studied the length of the game on trees, grids, planar graphs, and binomial random graphs when playing with ``extra'' cops.  Pisantechakool~\cite{Pis16} showed that for $n$-vertex planar graphs, three cops can always capture a robber within $2n$ rounds.  F\"{o}rster, Nuridini, Uitto, and Wattenhoffer~\cite{FNUW15} constructed graphs in which it takes a single cop $O(\ell \cdot n)$ time to capture $\ell$ robbers. 

In general, the capture time of an $n$-vertex graph with cop number $k$ must be $O(n^{k+1})$, and this is straightforward to prove -- see, for example, \cite{BI93}.  However, it is not so easy to determine whether this upper bound is asymptotically tight.  The aforementioned work of Bonato, Golovach, Hahn, and Kratochv\'{i}l \cite{BGHK09} and Gaven\v{c}iak \cite{Gav10} shows that in fact this bound is not tight when $k=1$.  Moreover, to the best of our knowledge, there are no examples in the literature of $n$-vertex graphs whose capture time is even superlinear in $n$, let alone on the order of $n^{k+1}$.  However, in this paper, we show that the simple upper bound on $\capt(G)$ is in fact tight for $k \ge 2$: our main result, Theorem~\ref{thm:main}, states that for fixed $k \ge 2$, the maximum value of the capture time among $n$-vertex graphs with cop number $k$ is $\Theta(n^{k+1})$.  (While this paper was in preparation, we were informed that Theorem~\ref{thm:main} has recently and independently been obtained by Brandt, Emek, Uitto, and Wattenhoffer in \cite{BEUW17} using a different argument.  However, we believe that the techniques used in our proof are of independent interest, particularly Theorem~\ref{thm:directed}: very little is known about Cops and Robbers played on directed graphs, and Theorem~\ref{thm:directed} establishes a close connection between the game played on undirected graphs and the game played on directed graphs.)  Additionally, we prove in Theorem~\ref{thm:directed_capture} that the maximum capture time among directed $n$-vertex graphs with cop number 1 is $\Theta(n^2)$, in contrast to the situation on undirected graphs.

To prove Theorem~\ref{thm:main}, we consider a generalized version of the game of Cops and Robbers.  Some formulations of the game do not allow players to remain still on their turns; rather, each cop and the robber must, on their corresponding turns, follow an edge.  When the game is presented in this manner, the underlying graph is typically assumed to be {\em reflexive}, meaning that each vertex has a {\em loop} (i.e. an edge joining that vertex to itself).  Consequently, this alternative formulation of the game is functionally equivalent to the original: instead of remaining still, an entity can simply follow the loop at its current vertex.  In this paper, we find it useful to use this alternate formulation of the game but to forego the assumption that the graph is reflexive.  Rather, some vertices may have loops, meaning that players can remain at that vertex indefinitely; others may have no loops, meaning that players at that vertex cannot remain there, and instead must move to an adjacent vertex.  We also permit our underlying graph to be directed; on a directed graph, players must follow edges in their prescribed directions.  (We ensure that the vertices of our graph always have at least one out-neighbor, so that we never reach a situation in which a player has no legal moves.)  Finally, we make use of a variant known as ``Cops and Robbers with Protection'', introduced by Mamino in~\cite{Mam13}.  In this variant, edges may be deemed ``protected''; cops may freely traverse protected edges, but can only capture the robber by reaching the robber's vertex along an unprotected edge. 

Our paper is laid out as follows.  In Section~\ref{sec:preliminaries}, we discuss the variants of Cops and Robbers that we will use throughout the paper.  In particular, we show how to model an instance of Cops and Robbers with Protection played on a directed graph using an instance of ordinary Cops and Robbers played on a reflexive undirected graph, without substantially changing the capture time.  We also outline the proof of our main result.  In Sections~\ref{sec:construction} and \ref{sec:main}, we prove our main result by showing how to construct $n$-vertex graphs with cop number $k$ and capture time $\Omega(n^{k+1})$.  Next, in Section~\ref{sec:directed}, we show that $n$-vertex strongly-connected directed graphs with cop number 1 may have capture time $\Omega(n^2)$.  Finally, in Section~\ref{sec:complexity}, we briefly discuss some consequences of our results in the area of computational complexity.

Throughout the paper, we denote the vertex set and edge set of a graph or directed graph $G$ by $V(G)$ and $E(G)$, respectively.  We denote an undirected edge joining vertices $u$ and $v$ by $uv$, and we denote a directed edge from $u$ to $v$ by $\overrightarrow{uv}$.  In the standard model of Cops and Robbers, we say that a cop {\em defends} a vertex $v$ when she occupies $v$ or some neighbor of $v$.  In other words, $v$ is defended if a robber moving to $v$ would be captured on the ensuing cops' turn.  The precise definition of the term will need to be modified slightly for some of the Cops and Robbers variants we consider, but the intuitive definition -- a vertex is defended if the robber cannot safely move there -- will remain unchanged.  In an instance of Cops and Robbers in which there are enough cops to guarantee a win, we say that the cops {\em play optimally} or use an {\em optimal strategy} if they play in a way that minimizes the length of the game (relative to the robber's actions).  Similarly, we say that the robber plays optimally if he plays in a way that maximizes the length of the game (relative to the cops' actions).  By a {\em configuration} of a game in progress, we mean the collection of the positions of all cops and the robber.  For more notation and background in graph theory, we direct the reader to~\cite{Wes01}.
\end{section}

\begin{section}{Preliminaries and Overview}\label{sec:preliminaries}

In this section, we lay the groundwork for a proof of Theorem~\ref{thm:main}.  

\newtheorem*{thm:main}{Theorem \ref{thm:main}}
\begin{thm:main}
For fixed $k \ge 2$, the maximum capture time of an $n$-vertex graph with cop number $k$ is $\Theta(n^{k+1})$.
\end{thm:main}
%\begin{duplicate}[Theorem~\ref{thm:main}]
%For fixed $k \ge 2$, the maximum capture time of an $n$-vertex graph with cop number $k$ is $\Theta(n^{k+1})$.
%\end{duplicate}

We begin with a short proof of an easy upper bound on $\capt(G)$.  (This result is well-known -- see for example \cite{BI93} -- but we include a proof here for the sake of completeness.)

\begin{prop}\label{prop:main_upper}
If $G$ is an $n$-vertex graph with cop number $k$, then $\capt(G) \le n \cdot \binom{n+k-1}{k} = O(n^{k+1})$.\looseness=-1
\end{prop}
\begin{proof}
Cops and Robbers is a deterministic game of perfect information and, at any point during the game, only the current configuration matters; the precise sequence of moves used to reach that configuration is irrelevant.  Consequently, under optimal play by both players, each player's moves are determined entirely by the present configuration of the game.  It is thus straightforward to see that under optimal play, no two cop turns begin with the same configuration.  Indeed, suppose that rounds $t_1$ and $t_2$ (with $t_1 < t_2$) both begin with some configuration $C$.  Since the players' moves depend only on the current configuration of the game, the same moves played on rounds $t_1, t_1+1, \dots, t_2-1$ would subsequently be repeated by both players.  Consequently, the game would return to configuration $C$ an unbounded number of times, contradicting the assumption that the cops are using a winning strategy.  The claim now follows by noting that there are $n \cdot \binom{n+k-1}{k}$ different configurations: there are $n$ ways to choose the robber's position, and there are $\binom{n+k-1}{k}$ ways to choose a multiset of $k$ cop positions.  Thus an optimally-played game involves at most $n \cdot \binom{n+k-1}{k}$ cop turns and hence at most $n \cdot \binom{n+k-1}{k}$ rounds.
\end{proof}

Note that for fixed $k$, we have
$$n \cdot \binom{n+k-1}{k} \le n \cdot \frac{(n+k-1)^k e^k}{k^k} = O(n^{k+1}),$$
so Proposition \ref{prop:main_upper} establishes the upper bound in Theorem~\ref{thm:main}.

Establishing the lower bound in Theorem~\ref{thm:main} is considerably more difficult.  The remainder of this section is devoted to giving a high-level overview of the proof, while the proof itself appears in Section~\ref{sec:main}.  

For each integer $k \ge 2$, we aim to construct arbitrarily large graphs $G$ such that $\cnum(G) = k$ and $\capt(G) \ge c_k \cdot \size{V(G)}^{k+1}$ (where $c_k$ is a constant depending only on $k$).  One difficulty in doing this is that the game of Cops and Robbers can get quite complex.  There are myriad ways for the game to proceed, so finding an optimal cop or robber strategy for some graph -- or even just determining the capture time -- can be frustratingly complicated.  

To deal with this difficulty, we consider a generalization of the usual game of Cops and Robbers that gives us more power in constructing a ``game board'' that limits the players' freedom.  In Section~\ref{sec:main}, we use this model to produce graphs with high capture time (under the rules of this generalization).  Moreover, we prove that this is ``good enough'' to establish Theorem~\ref{thm:main}: we show that any instance of the game with large capture time under this new model can be transformed into an instance of the usual game of Cops and Robbers with large capture time and with the same cop number.

In fact, we consider two successive generalizations of the game.  The first of these is the game of ``Cops and Robbers with Protection,'' introduced by Mamino~\cite{Mam13}.  In this variant, each edge of the underlying graph $G$ is designated either {\em protected} or {\em unprotected}.  Additionally, the winning conditions of the game are slightly different: the cops win only if some cop reaches the same vertex as the robber by traveling along an unprotected edge.  (The cops may still traverse protected edges, but cannot use them to directly capture the robber.)  One subtle consequence of this change is that the robber may, in this variant, move to the same vertex as a cop without immediately losing the game.  (Note, however, that if this vertex has an unprotected loop, then the cop may use it to win on the ensuing cop turn.)  

We refer to a vertex with an unprotected loop (resp. protected loop) as an {\em unprotected vertex} (resp. {\em protected vertex}).  Vertices without loops are considered neither protected nor unprotected.  We say that a cop {\em defends} a vertex if it is adjacent to that vertex along an unprotected edge.  A graph together with a specification of protectedness or unprotectedness for each edge is a {\em protected graph}.  We define the {\em cop number} and {\em capture time} of a protected graph in this new game analogously as in the original game.

Mamino showed that Cops and Robbers with Protection can be reduced to ordinary Cops and Robbers.  He did this by showing how to construct, for every protected graph $G$, an unprotected graph $H$ with the same cop number and with $\size{V(H)}$ not ``too much larger'' than $\size{V(G)}$.  While Mamino was not concerned with capture time, his proof shows that in fact $\capt(H) = \capt(G)$.  Thus, close analysis of Mamino's construction yields the following:  

\begin{lemma}[essentially \cite{Mam13}, Lemma 3.1]\label{lem:protected}
For every protected graph $G$, there exists a graph $H$ such that $\cnum(H) = \cnum(G)$, $\capt(H) = \capt(G)$, and $\size{V(H)} < 4 \cdot \cnum(G)^2 \cdot \size{V(G)}$.
\end{lemma}

Next, we show that instead of restricting ourselves to reflexive undirected graphs, we can consider (not necessarily reflexive) directed graphs.  To be precise, we prove the following:

\newtheorem*{thm:directed2}{Theorem \ref{thm:directed}}
\begin{thm:directed2}
Fix $k \ge 2$, and let $G$ be a protected (not necessarily reflexive) directed graph with $\cnum(G) = k$ such that every vertex in $G$ has at least one in-neighbor and one out-neighbor.  If $H$ is constructed from $G$ as specified above, then $\cnum(H) = k$ and $\captdir(G) + 1 \le \capt(H) \le \captdir(G) + 2$.  In addition, $\size{V(H)} = (3k+3)\size{V(G)} + 8k + 3$.
\end{thm:directed2}
%\begin{duplicate}[Theorem~\ref{thm:directed}]
%Fix $k \ge 2$, and let $G$ be a protected directed graph with $\cnum(G) = k$ such that every vertex in $G$ has at least one in-neighbor and one out-neighbor.  There exists a reflexive protected undirected graph $H$ such that $\cnumprot(H) = k$ and $\captdir(G) + 2 \le \capt(H) \le \captdir(G) + 5$.  In addition, $\size{V(H)} = (3k+3)\size{V(G)} + 8k + 3$.
%\end{duplicate}

After establishing Theorem~\ref{thm:directed}, we are ready to prove Theorem~\ref{thm:main}.  We aim to construct reflexive undirected graphs with large capture time; Theorem~\ref{thm:directed} shows that it suffices to construct (not necessarily reflexive) protected directed graphs with large capture time.  In particular, for $k \ge 2$ we show how to construct arbitrarily large protected directed graphs $G$ with with $\cnumdir(G) = k$ and $\captdir(G) \ge \left(\frac{\size{V(G)}}{2k}\right)^{k+1}$.  Once we have constructed some such directed graph, we can use Theorem~\ref{thm:directed} to find a corresponding protected undirected graph.  Finally, Lemma~\ref{lem:protected} gives a corresponding ordinary undirected graph, which completes the proof of Theorem~\ref{thm:main}.

\section{Modeling Directed Graphs with Undirected Graphs}\label{sec:construction}

To prove Theorem~\ref{thm:directed} we must show how to construct, for a given protected directed graph $G$, a corresponding undirected protected graph $H$ with the same cop number, roughly the same capture time, and not ``too many more'' vertices.  Since this construction is rather involved, we devote this section to an explanation of the construction, leaving the actual proof of Theorem~\ref{thm:directed} for Section~\ref{sec:main}.

Let $G$ be a protected directed graph and let $k = \cnumdir(G)$.  We construct a reflexive protected undirected graph $H$ as follows. The vertex set of $H$ is comprised of sets $S$, $T_0$, $T_1$, $T_2$, $C_0$, $C_1, C_2, C^*, R_0, R_1, R_2$, and $R^*$.  We refer to vertices in $C_0$, $C_1$, and $C_2$ as {\em cop vertices}, while vertices in $R_0$, $R_1$, and $R_2$ are {\em robber vertices}.  $S \cup T_0 \cup T_1 \cup T_2$ is referred to as the {\em reset clique}, with $S$ itself comprising the {\em core} and $T_i$ comprising the {\em $i$th wing} of the clique.  Vertices in $C^*$ are {\em cop starter vertices}, while vertices in $R^*$ are {\em robber starter vertices}.  (See Figure~\ref{fig:directed_overview}.)  All vertices in the reset clique are protected, while all other vertices in $H$ are unprotected.  Throughout the construction, the indices on the $C_i$ and $R_i$ should be taken modulo 3 when necessary; for example, when $i=2$, the expression ``$C_{i+1}$'' refers to $C_0$.

\begin{figure}[hb]
\begin{center}
\includegraphics{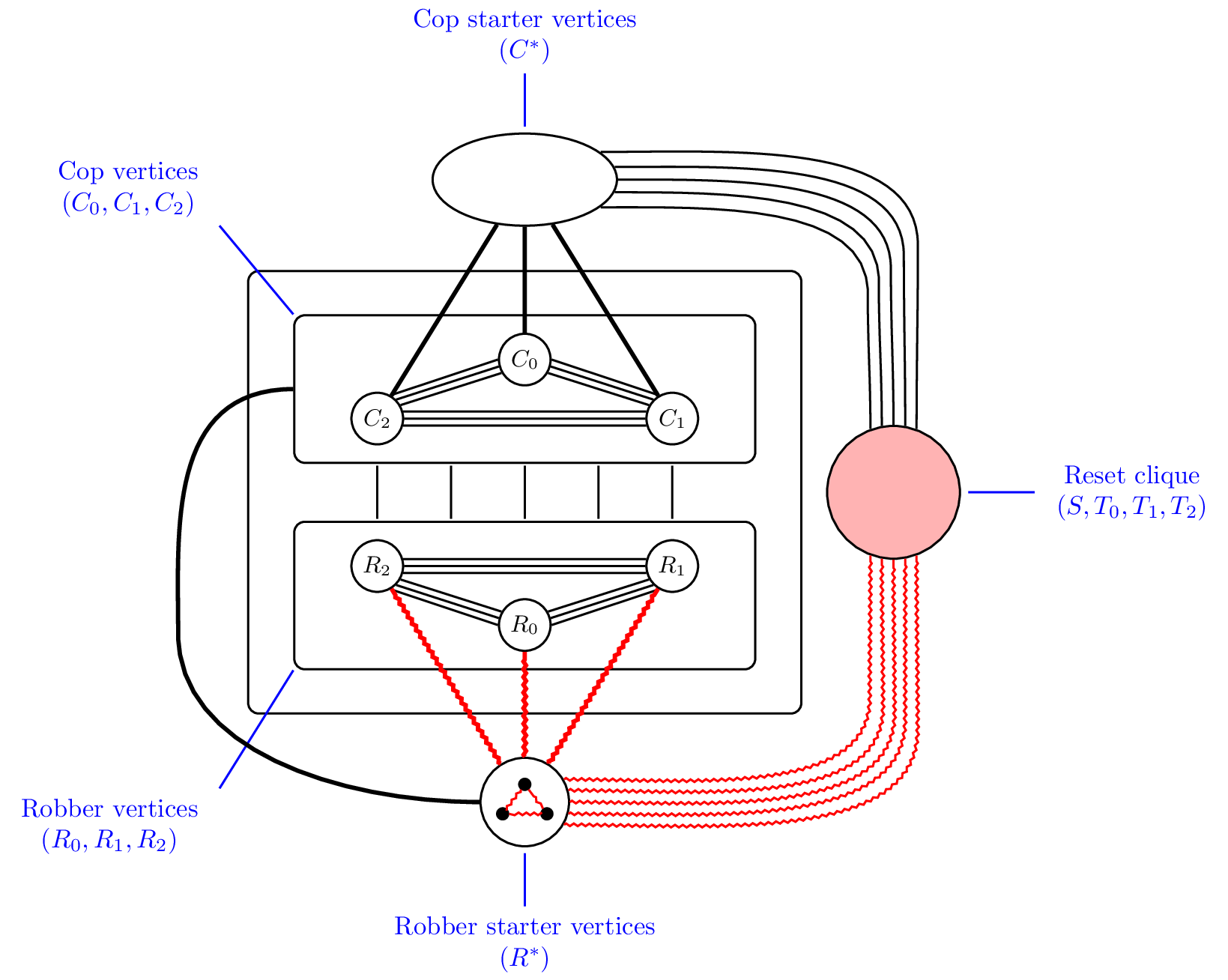}
\end{center}
\caption{: cop vertices and robber vertices.  Jagged edges are protected.  A thick edge indicates that all possible edges of this sort are present.}\label{fig:directed_overview}
\end{figure}

The core of the reset clique contains $4k$ vertices, namely $s_0, s_1, \dots, s_{4k-1}$.  In addition, for $i \in \{0,1,2\}$, the set $T_i$ contains $k$ vertices, namely $t_0^i, t_1^i, \dots, t_{k-1}^i$.  Every pair of vertices in the reset clique is joined by a protected edge.  Under the proper circumstances, the reset clique will permit the robber to ``reset'' the game back to its initial state.  As with the $C_i$ and $R_i$, indices on vertices within the $T_i$ should be taken modulo $k$ and indices on vertices in $S$ should be taken modulo $4k$.

The sets $C_0$, $C_1$, and $C_2$ each contain $k$ copies of every vertex in $G$.  For $v \in V(G)$, $i \in \{0, 1, 2\}$, and $j \in \{0, 1, \dots, k-1\}$, we denote by $\kappa(v;i,j)$ the $j$th copy of $v$ belonging to $C_i$.  Within each $C_i$, the copies of a given vertex form a clique; that is, for all $v, i, j,$ and $j'$, the vertices $\kappa(v;i,j)$ and $\kappa(v;i,j')$ are joined by an unprotected edge.  Aside from these edges and loops, there are no edges with both endpoints in any $C_i$.  The sets $R_0$, $R_1$, and $R_2$ each contain one copy of every vertex in $G$, and each $R_i$ is independent.  We denote by $\rho(v;i)$ the copy of $v$ in $R_i$.  Our intent is that for the bulk of the game, the cops occupy vertices within the $C_i$ (that is, ``cop vertices''), while the robber occupies a vertex within one of the $R_j$ (that is, a ``robber vertex'').

We seek to construct $H$ so that play of the game on $H$ mirrors play on $G$, in the sense that we can equate a cop occupying $\kappa(v;i,j)$ in $H$ with a cop occupying $v$ in $G$; likewise, the robber occupying $\rho(v;i)$ in $H$ is analogous to the robber occupying $v$ in $G$.  We also aim to greatly restrict the flexibility that both players enjoy.  In particular, our intent is that under any optimal cop strategy, if the cops are positioned within $C_i$ after some cop turn, then the robber must be positioned within $R_i$.  Moreover, we aim to force the cops to move from $C_0$ to $C_1$ to $C_2$ to $C_0$ and so on; forcing the cops to keep moving in the ``forward direction'' among the $C_i$ will allow us to simulate playing on the directed graph $G$.  Likewise, we aim to force the robber to move from $R_0$ to $R_1$ to $R_2$ to $R_0$ and so on.

We now add edges among the cop and robber vertices.  For all $\overrightarrow{uv} \in E(G)$, all $i \in \{0, 1, 2\}$, and all $j \in \{0, 1, \dots, k-1\}$, add unprotected edges joining $\kappa(u;i,j)$ to $\kappa(v;i+1,j)$ and joining $\rho(u;i)$ to $\rho(v;i+1)$.  (Note that $u$ and $v$ need not be distinct.)  These edges ensure that each ``forward'' movement by a cop or robber (from $C_i$ to $C_{i+1}$ or from $R_i$ to $R_{i+1}$) corresponds to following an edge from $G$ in the forward direction.  If $\overrightarrow{uv}$ is unprotected, then we also add an unprotected edge joining $\kappa(u;i,j)$ to $\rho(v;i+1)$.  These edges allow a cop to capture the robber in $H$ provided that she would be able to do so in $G$.  (See Figure~\ref{fig:crvertices_overview}.)

\begin{figure}[hb]
\begin{minipage}{0.49\textwidth}
\begin{center}
\includegraphics{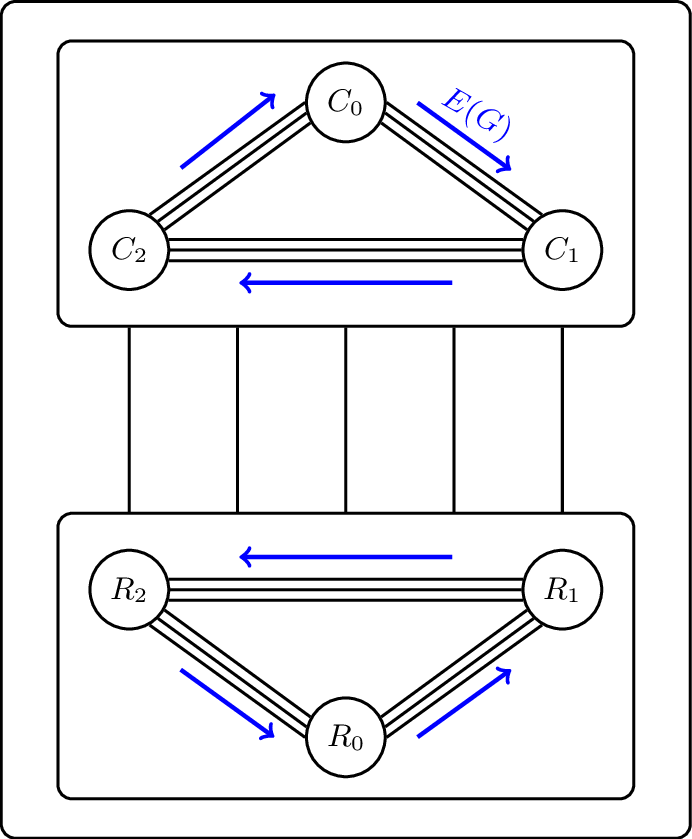}
\end{center}
\end{minipage}
%\quad
\begin{minipage}{0.5\textwidth}
%\begin{itemize}
Each $R_i$ contains one copy of each vertex in $G$; each $C_i$ contains $k$ copies.

\medskip\medskip

For $\overrightarrow{uv} \in E(G)$, every copy of $u$ in $C_i$ is adjacent to every copy of $v$ in $C_{i+1}$, and the copy of $u$ in $R_i$ is adjacent to the copy of $v$ in $R_{i+1}$.

\medskip\medskip\medskip\medskip

Not pictured:
\medskip
\begin{itemize}
\item For $\overrightarrow{uv} \in E(G)$, each copy of $u$ in $C_i$ is adjacent to the copy of $v$ in $R_{i+1}$.
\medskip
\item Every vertex in $C_i$ is adjacent to all vertices in $R_{i-1} \cup R_i$.
\medskip
\item The $j$th copy of $v$ in $C_i$ is adjacent to $t_j^i$, $s_{4j+i}$, $s_{4j+i+1}$, $s_{4j+i+2}$, and $s_{4j+i+3}$. 
\medskip
\item The copy of $v$ in $R_i$ is adjacent to all of $S \cup T_i$ along protected edges.
\end{itemize}
\end{minipage}
\caption{: cop vertices and robber vertices.}\label{fig:crvertices_overview}
\end{figure}

Next, for $i \in \{0,1,2\}$, add protected edges joining each vertex of $R_i$ to each vertex of $S \cup T_i$.  These edges permit the robber to ``escape'' to the reset clique, should the cops fail to adequately defend it.  Additionally, add unprotected edges joining each $\kappa(v;i,j)$ to $s_{4j+i}$, $s_{4j+i+1}$, $s_{4j+i+2}$, $s_{4j+i+3}$, and $t_{j}^i$.  These edges permit cops in $C_i$ to defend the core and $i$th wing of the reset clique, but only by positioning themselves in very special ways (see below).  We also add unprotected edges joining all vertices of $C_i$ to all vertices of $R_i \cup R_{i-1}$.  These edges force the robber to keep moving ``forward'' from $R_0$ to $R_1$ to $R_2$ to $R_0$ and so forth, in order to stay one step ``ahead'' of the cops.

The set $C^*$ of cop starter vertices contains $k$ vertices, namely $c^*_0, c^*_1, \dots, c^*_{k-1}$.  For $j \in \{0, 1, \dots, k-1\}$, we add unprotected edges joining $c^*_j$ to $s_{4j+3},s_{4j+4},s_{4j+5},s_{4j+6}, t_j^0,t_j^1$, and $t_j^2$.  We also add unprotected edges joining every cop starter vertex to every cop vertex and every robber vertex.  Finally, add unprotected edges joining all robber starter vertices to all cop vertices and joining each $r^*_i$ to all vertices in $R_{i+1}$, along with protected edges joining each pair of robber starter vertices and protected edges joining each $r^*_i$ to all vertices in the core and $i$th wing of the reset clique.  These edges ensure that when the cops choose to start the game by occupying all of the cop starter vertices, the robber must start on one of the robber starter vertices.  (In fact, under optimal play the cops must start on the cop starter vertices, but this is less clear; see Section~\ref{sec:main}.)

\section{Main Results}\label{sec:main}

We are now ready to prove Theorems~\ref{thm:directed} and \ref{thm:main}.  To prove Theorem~\ref{thm:directed}, we must show that $\cnumprot(H) = \cnumdir(G) = k$ and that $\captprot(H)$ is roughly equal to $\captdir(G)$ (where $G$ is any given protected directed graph and $H$ is the protected undirected graph constructed from $G$ in Section~\ref{sec:construction}).  We begin by describing how we ``expect'' the game on $H$ to be played. 

To simplify the analysis, we introduce some additional terminology.  For all $v \in V(G)$ and $i \in \{0,1,2\}$, we refer to $\kappa(v;i,j)$ as a {\em $j$-vertex}; $c^*_j$ is also considered a $j$-vertex.  In the game on $H$, we say that the cops occupy a {\em stable position} if either the cops occupy all $k$ cop starter vertices, or all cops occupy vertices within $C_i$ for some $i$ with one cop on a $j$-vertex for all $j \in \{0, 1, \dots, k-1\}$.  We say that the game is in a {\em canonical configuration} if the cops occupy a stable position in some $C_i$ and either:
\begin{itemize}
\item It's the cops' turn and the robber occupies a vertex in $R_{i+1}$, or
\item It's the robber's turn and the robber occupies a vertex in $R_i$. 
\end{itemize}

As we will see below, the cops will want to prevent the robber from ever reaching the reset clique.  To do this, they must ensure that they always defend the robber's neighbors in the clique; this requires them to position themselves carefully.  Recall that we say a cop {\em defends} a vertex $v$ when there is an unprotected edge joining the cop's current vertex with $v$.

\begin{lemma}\label{lem:stable}
The cops defend all vertices of the core of the reset clique in $H$ if and only if they occupy a stable position.  Moreover, if the cops occupy a stable position, then they defend the $i$th wing of the reset clique if and only if they occupy either $C_i$ or the cop starter vertices.
\end{lemma}
\begin{proof}
It is clear from the construction of $H$ that if the cops occupy a stable position, then they defend all vertices of the core.  Suppose now that the cops defend all vertices of the core.  Every cop vertex and cop starter vertex defends exactly four vertices within the core, while no other vertices in the graph defend any vertices of the core.  Since there are $4k$ vertices in the core and only $k$ cops, all $k$ cops must occupy cop vertices or cop starter vertices and, moreover, no two cops can have a common neighbor in the core.  By construction, if any two cops both occupy a $j$-vertex for some $j$, then they have a common neighbor (namely $s_{4j+3}$); consequently, we must have one cop on a $j$-vertex for all $j \in \{0, 1, \dots, k-1\}$.  Finally, by symmetry, it suffices to show that if one cop occupies a vertex of $C_0$, then so must the other $k-1$ cops.  Suppose otherwise, and choose $\ell$ so that the cop occupying an $\ell$-vertex sits in $C_0$, while the cop occupying an $(\ell-1)$-vertex does not.  By construction, these two cops both defend $s_{4\ell}$, so the cops cannot defend the entire core.  

The second half of the claim is clear.
\end{proof}

An important consequence of Lemma~\ref{lem:stable} is that when the robber resides in the reset clique, the cops can force him to leave only by occupying all $k$ cop starter vertices.  When they do so, the robber must move to one of the robber starter vertices, lest he be captured on the cops' next turn.  Thus motivated, we define an {\em initial configuration} to be a game state in which the cops occupy all $k$ cop starter vertices, the robber occupies a robber starter vertex, and it is the cops' turn.

\begin{lemma}\label{lem:copvx}
Suppose that the cops occupy a stable position and it is the robber's turn.  If the robber moves to any cop vertex, then the cops can capture him within the two subsequent rounds.
\end{lemma}
\begin{proof}
Suppose the robber moves to a $j$-vertex within $C_i$.  In response, the cop currently on a $j$-vertex moves to any $j$-vertex in $C_i$, while any other cop moves to a cop starter vertex.  (Note that this is always possible because each vertex in $G$ has at least one in-neighbor and one out-neighbor, so no matter which cop vertex a cop occupies, she always has at least one $j$-vertex neighbor in each of $C_0$, $C_1$, and $C_2$.)  The cop on the cop starter vertex now defends all cop and robber vertices, while the cop on a $j$-vertex in $C_i$ defends all robber starter vertices and the robber's neighbors in the reset clique.  Thus, no matter how the robber moves, the cops can win on their next turn.
\end{proof}

\begin{lemma}\label{lem:initial}
Suppose that the game is in an initial configuration with the robber on $r^*_i$.  Under optimal play by both players, either the robber loses within the next three rounds, or:
\begin{description}
\item[(1)] After one more round, the game reaches a canonical configuration with the cops in $C_i$ and the robber in $R_{i+1}$, and
\item[(2)] For the remainder of the game, the robber never moves to an undefended vertex in the reset clique.
\end{description}
Moreover, if the game reaches a canonical configuration as in (1), then the robber may occupy whichever vertex of $R_{i+1}$ he chooses.
\end{lemma}
\begin{proof}
We begin with claim (2).  If the robber ever moves into the reset clique, then by Lemma~\ref{lem:stable} and the ensuing discussion, the game must eventually return to an initial configuration.  All initial configurations are equivalent up to symmetry, so the rounds leading up to this second initial configuration have served no purpose for the cops.  Thus, under optimal play by the cops, the game never returns to an initial configuration, so the robber must never reach the reset clique. 

We next turn to claim (1).  If the cops all remain on the cop starter vertices, then the robber can simply remain on $r^*_i$; this is clearly suboptimal for the cops.  Otherwise, if the cops do not move to a stable position in $C_i$, then by Lemma~\ref{lem:stable}, they leave some vertex $v$ of $S \cup T_i$ undefended.  Thus the robber can, on his ensuing turn, move to the reset clique; by claim (2), this cannot happen under optimal play.  Suppose therefore that the cops move to a stable position within $C_i$.  By Lemma~\ref{lem:copvx}, the robber cannot move to a cop vertex without being captured in short order.  The cops defend all of the robber's neighbors in the reset clique, so the robber cannot move there, either; likewise, he cannot remain in place or move to a different robber starter vertex.  The only remaining option is for the robber to move to $R_{i+1}$, resulting in a canonical configuration of the desired type; since $r^*_i$ is adjacent to all of $R_{i+1}$, the robber may occupy any vertex of $R_{i+1}$ he chooses.
\end{proof}

Recall that one of our main goals in constructing $H$ is to greatly restrict the freedom enjoyed by both players.  Claim (2) of the lemma above shows that the cops cannot allow the robber to safely enter the reset clique; this is the primary means by which we restrict the movements of the cops.  Conversely, the robber's movements are restricted by the threat of capture.  As we next show, these restrictions are severe enough that (under optimal play) the game is nearly always in a canonical configuration.

\begin{lemma}\label{lem:canonical}
Suppose that, on a cop turn, the game is in a canonical configuration with the cops in $C_{\ell}$ and the robber in $R_{\ell+1}$.  Under optimal play, either the robber loses within the next three rounds, or the game proceeds to a canonical configuration with the cops in $C_{\ell+1}$ and the robber in $R_{\ell+2}$.  
\end{lemma}
\begin{proof}
If some cop currently defends the robber's vertex, then the cops win immediately; suppose otherwise.  Since the robber occupies a vertex in $R_{\ell+1}$, his current vertex is adjacent to all of $S \cup T_{\ell+1}$.  Thus, by Lemma~\ref{lem:stable}, unless the cops move to a stable position in $C_{\ell+1}$, the robber can safely move to a vertex in the reset clique.  By Lemma~\ref{lem:initial}, this cannot happen under optimal play. 

We may thus suppose that the cops move to a stable position in $C_{\ell+1}$.  The cops now defend all of $R_{\ell}$, all of $R_{\ell+1}$, the robber starter vertices, and all of the robber's neighbors in the reset clique.  Moreover, by Lemma~\ref{lem:copvx}, if the robber moves to a cop vertex, then he can be captured within the following two rounds.  The robber's only remaining option is to move into $R_{\ell+2}$, resulting in a canonical configuration of the desired form.
\end{proof}

We are finally ready to prove Theorem~\ref{thm:directed}.

\begin{theorem}
\label{thm:directed}
Fix $k \ge 2$, and let $G$ be a protected (not necessarily reflexive) directed graph with $\cnum(G) = k$ such that every vertex in $G$ has at least one in-neighbor and one out-neighbor.  If $H$ is constructed from $G$ as specified above, then $\cnum(H) = k$ and $\captdir(G) + 1 \le \capt(H) \le \captdir(G) + 2$.  In addition, $\size{V(H)} = (3k+3)\size{V(G)} + 8k + 3$.
\end{theorem}
\begin{proof}
It is clear from the construction of $H$ that $\size{V(H)} = (3k+3)\size{V(G)} + 8k + 3$.  To establish the rest of the claim, we show that the cops can win the game on $H$ within $\captdir(G)+2$ rounds and that the robber can evade capture on $H$ for at least $\captdir(G)$ turns.  As shown in Lemmas~\ref{lem:initial} and \ref{lem:canonical}, optimal play ensures that the game on $H$ will typically be in a canonical configuration.  We associate each canonical configuration in $H$ with a configuration of the game in $G$ in the natural way.  Each of the $k$ cops in $H$ occupies some cop vertex $\kappa(v;i,j)$; we view this cop as occupying the vertex $v$ in $G$.  Likewise, when the robber occupies some vertex $\rho(w;i)$ in $H$, we imagine that he occupies vertex $w$ in $G$.

We begin by giving a cop strategy to capture the robber on $H$ within $\captdir(G) + 2$ rounds.  Throughout the game on $H$, the cops will imagine playing a game on $G$ and use their strategy on $G$ to guide their play on $H$.  The cops begin by occupying all $k$ cop starter vertices.  To avoid immediate capture, the robber must then occupy one of the robber starter vertices, say (without loss of generality) $r^*_0$.  The cops now turn to the game on $G$ and choose their starting positions in that game.  For convenience, index the cops from 0 to $k-1$.  If cop $j$ occupies vertex $v$ in $G$, then she moves to vertex $\kappa(v;0,j)$ in $H$.  (This is always possible because each cop starter vertex is adjacent to all vertices in $C_0$.)  Thus, in $H$, the cops move to a stable position within $C_0$ that corresponds to their choice of initial positions on $G$.  As argued in Lemma~\ref{lem:initial}, if the robber is to survive for more than three more rounds, then he must occupy some vertex $\rho(w;1)$.  The cops now imagine that, in the game on $G$, the robber has chosen to occupy vertex $w$.

The game on $H$ has now entered a canonical configuration.  The cops imagine their next move (under optimal play) in $G$ and mirror this move in $H$, while simultaneously moving to a stable position within $C_1$.  In particular, if cop $j$ moves from $v$ to $w$ in $G$, then she moves from $\kappa(v;0,j)$ to $\kappa(w;1,j)$ in $H$.  (Note that since $\overrightarrow{vw} \in E(G)$, vertices $\kappa(v;0,j)$ and $\kappa(w;1,j)$ are adjacent in $H$.)  As argued in Lemma~\ref{lem:canonical}, either the robber moves to $R_2$ or he loses within the following two rounds.  Suppose the latter.  Since the game on $G$ has not yet ended, it has lasted at most $\captdir(G)-1$ rounds.  One round was played in the game on $H$ before the game on $G$ even began and as many as two more rounds might yet be played.  In total, the game on $H$ lasts at most $\captdir(G)+2$ rounds.  Now suppose instead that the robber moves from his current position $\rho(x;1)$ to some vertex $\rho(y;2)$ in $R_2$.  By construction, vertices $\rho(x;1)$ and $\rho(y;2)$ are adjacent in $H$ only if $\overrightarrow{xy} \in E(G)$.  Thus, in the game on $G$, the cops may imagine that the robber has moved from $x$ to $y$.  The game on $H$ remains in a canonical configuration and, moreover, this configuration corresponds to the configuration of the game on $G$.

The game on $H$ continues in this manner until either the robber fails to move to a canonical configuration as outlined in Lemma~\ref{lem:canonical} or until the cops capture the robber on $G$.  In the former case, as argued above, the game on $H$ lasts at most $\captdir(G)+2$ rounds.  In the latter case, some cop $j$ has followed an unprotected edge in $G$ from her vertex $v$ to the robber's vertex $w$ (where $v$ and $w$ are not necessarily distinct).  (Recall that in the game of Cops and Robbers with Protection, this is the only way that the game can end; in particular, unlike in ordinary Cops and Robbers, the game does not end if the robber moves to the cop's current vertex.)  In this case, in $H$, cop $j$ presently occupies $\kappa(v;i,j)$ for some $i$, while the robber occupies $\rho(w;i+1)$.  Since $\overrightarrow{vw}$ is an unprotected edge in $G$, these two vertices are joined in $H$ by an unprotected edge, so cop $j$ may proceed to capture the robber in $H$.  Since at most $\capt(G)$ rounds have elapsed in $G$ and one additional round was played in $H$ before the game in $G$ even began, in total at most $\capt(G)+1$ rounds have been played in $H$.

To show that $\capt(H) \ge \captdir(G)+1$, we use a similar argument, except that this time we give a strategy for the robber.  We assume throughout that the cops play optimally on $H$.  At the outset of the game on $H$, there are two possibilities: either the cops begin by occupying all $k$ cop starter vertices, or they don't.  In the latter case, by Lemma~\ref{lem:stable}, some vertex of the reset clique remains undefended; the robber chooses to begin there.  The robber can henceforth remain in the clique until the cops do occupy all $k$ cop starter vertices, at which point the robber moves to $r^*_0$.  If instead the cops begin by occupying the cop starter vertices, the robber simply begins on $r^*_0$; since this is clearly a more efficient line of play for the cops, we may suppose that this is what happens.

Thus, the game begins in an initial configuration with the robber on $r^*_0$.  By Lemma~\ref{lem:initial}, the cops must always defend all of the robber's neighbors in the reset clique.  They can do this only by moving to a stable position within $C_0$.  As above, we can associate this stable position with an initial position for the cops in the game on $G$.  The robber now considers the game on $G$ and chooses his initial position in that game; say he decides to begin on vertex $v$.  In the game on $H$, the robber moves to $\rho(v;1)$.  The game on $H$ has now entered a canonical configuration that corresponds to the current configuration of the game on $G$.  

It is now the cops' turn.  The cops all occupy vertices of the form $\kappa(u;0,j)$, while the robber occupies some vertex $\rho(v;1)$.  By construction, these two vertices are adjacent along an unprotected edge if and only if $\overrightarrow{uv}$ is an unprotected edge in $G$.  Thus, some cop can capture the robber on their ensuing turn on $H$ if and only if she can do so on $G$.  Otherwise, to prevent the robber from reaching the reset clique, the cops must move to a stable position in $C_1$.  As before, this cop movement in $H$ corresponds to a legal cop movement in $G$.  The robber imagines that the cops have played thus on $G$, decides which vertex to move to in that game, and moves (in $H$) to the corresponding vertex in $R_2$.  As above, the game continues in this manner, with each player's moves in $H$ corresponding to legal moves in $G$.  Eventually, the cops capture the robber in $H$.  By construction, the cops cannot capture the robber on $H$ until such time as they can also capture him on $G$.  Since the robber plays optimally in $G$, this takes at least $\captdir(G)$ rounds; since one additional round was played in $H$, we have $\capt(H) \ge \captdir(G)+1$, as claimed.
\end{proof}

Armed with Theorem~\ref{thm:directed}, we are ready to prove Theorem \ref{thm:main}.

\begin{theorem}\label{thm:main}
For fixed $k \ge 2$, the maximum capture time of an $n$-vertex graph with cop number $k$ is $\Theta(n^{k+1})$.
\end{theorem}
\begin{proof}
It follows from Proposition \ref{prop:main_upper} that the capture time of an $n$-vertex graph with cop number $k$ is $O(n^{k+1})$, so it suffices to establish a matching lower bound.  In particular, we will show that there exist arbitrarily large graphs $H$ with cop number $k$ and capture time at least $\left( \frac{\size{V(H)}}{40k^4}\right )^{k+1}$.

We first show how to construct a protected directed graph $G$ with $\cnumdir(G) = k$ and $\captdir(G) \ge \left(\frac{n}{2k}\right)^{k+1}$, where $n = \size{V(G)}$.  By Theorem~\ref{thm:directed}, it then follows that there exists a protected reflexive undirected graph $G'$ with $\cnumprot(G') = k$, $\captprot(G') \ge \left(\frac{n}{2k}\right)^{k+1},$ and $\size{V(G')} = (3k+3)n+8k+3$.  Finally, Lemma~\ref{lem:protected} implies the existence of a reflexive undirected graph $H$ with $\cnum(H) = k$, $\capt(H) \ge \left(\frac{n}{2k}\right)^{k+1}$, and $\size{V(H)} < 4k^2 \size{V(G')} < 20k^3n$ (for sufficiently large $n$).  Thus, $\capt(H) \ge \left( \frac{\size{V(H)}}{40k^4}\right )^{k+1}$, as claimed.

Our goal in constructing $G$ is to restrict the cops' actions so greatly that they have only one reasonable line of play -- a line of play that happens to take a long time to resolve.  This will greatly simplify the analysis of the game. 

The vertex set of $G$ consists of five parts: $S$, the {\em reset clique}; $C_0, C_1, \dots, C_{k-1}$, the {\em cop tracks}; $R$, the {\em robber track}; $X$, the set of {\em escape vertices}; and one special vertex $\omega$.  The reset clique consists of the $k$ vertices $s_0, s_1, \dots, s_{k-1}$.  Each cop track $C_i$ consists of the $q_i$ vertices $c_{i,0}, c_{i,1}, \dots, c_{i,q_i-1}$, where the $q_i$ will be specified later.  Likewise, the robber track consists of the $p-1$ vertices $r_0, \dots, r_{p-2}$ (where $p$ will be specified later) along with the $k$ vertices $r_{p-1}^0, r_{p-1}^1, \dots, r_{p-1}^{k-1}$.  The set $X$ consists of the $k$ escape vertices $x_0, x_1, \dots, x_{k-1}$.  The vertices in the reset clique are reflexive and protected, $\omega$ is reflexive and unprotected, and all other vertices are irreflexive.

Between each pair of vertices $s,t$ in the reset clique, we add edges $\vec{st}$ and $\vec{ts}$, both protected.  For each $i \in \{0, 1, \dots, k-1\}$, we add an unprotected edge from $c_{i,0}$ to $s_i$.  On the cop tracks, we add unprotected edges from $c_{i,j}$ to $c_{i,j+1}$ for all $j$ (where $j$ is taken modulo $q_i$ where appropriate).  On the robber track, we add unprotected edges from $r_j$ to $r_{j+1}$ for $j \in \{0, 1, \dots, p-3\}$, along with unprotected edges from $r_{p-2}$ to $r_{p-1}^0, r_{p-1}^1, \dots, r_{p-1}^{k-1}$ and from each of $r_{p-1}^0, r_{p-1}^1, \dots, r_{p-1}^{k-1}$ to $r_0$.  We also add unprotected edges from every vertex on the robber track to every escape vertex, from every escape vertex to every vertex in the reset clique, and from every vertex in the reset clique to $r_0$.  Finally, we add unprotected edges from every vertex in $C_i$ to the escape vertex $x_i$, from each $c_{i,q_i-1}$ to $r_{p-1}^i$, and from $\omega$ to all vertices {\em except} those in the reset clique.  (Refer to Figure~\ref{fig:construction}.)

\begin{figure}[hb]
\begin{center}
\includegraphics{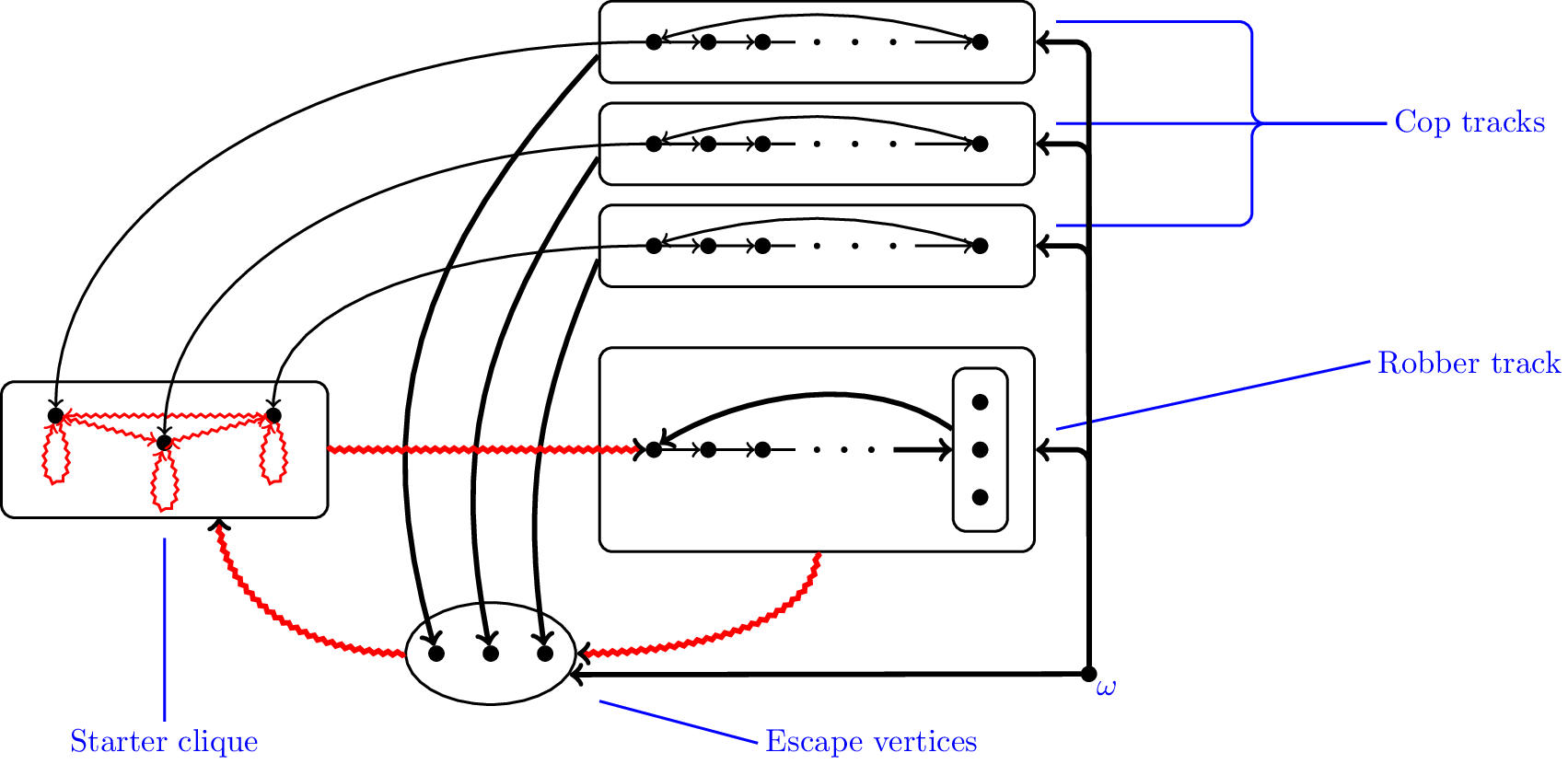}
\end{center}
\caption{: the graph $G$ (with $k=3$).  Jagged edges are protected.  A thick edge from one set of vertices to another indicates that all possible edges of this sort are present.}\label{fig:construction}
\end{figure}

Before proceeding, we make four observations.  First: the cops can defend the reset clique only by occupying all of $c_{0,0}, c_{1,0}, \dots, c_{k-1,0}$.  Second: a cop can enter $\omega$ only by beginning the game there; once he leaves, he can never return.  Third: the cops defend all $k$ escape vertices if and only if either some cop occupies $\omega$, or each cop track contains a cop.  When each cop track contains a cop, we say that the cops occupy a {\em stable position}.  Fourth: from a stable position, if the cop on $C_i$ ever leaves $C_i$, then he can never return; consequently, the cops can never again occupy all of $c_{0,0}, c_{1,0}, \dots, c_{k-1,0}$, so they can never again defend the reset clique.  Thus if, on the robber's turn, the cops ever fail to occupy a stable position (and no cop occupies $\omega$), then the robber can move to some undefended escape vertex and subsequently to the reset clique, whence he can never be captured.  Thus the cops must occupy a stable position on every robber turn or else lose the game.

At this point, we remark that because $k$ cops are needed to defend the reset clique, we have $\cnumdir(G) \ge k$.  We explain later why $\cnum(G) \le k$.
 
We are now ready to outline the robber's strategy for surviving ``long enough'' against $k$ cops.  At the beginning of the game, if the cops' initial placement leaves any vertex of the reset clique undefended, then the robber begins on some such vertex; he moves between undefended vertices in the clique until the cops occupy all of $c_{0,0}, c_{1,0},\dots, c_{k-1,0}$, at which point he moves to $r_0$.  Otherwise, the cops must have begun the game on precisely these vertices, so the robber begins on $r_0$.  In either case, the game reaches a configuration in which the cops occupy $c_{0,0}, c_{1,0}, \dots, c_{k-1,0}$, the robber occupies $r_0$, and it is the cops' turn; we refer to this as the {\em initial configuration} of the game.  In an initial configuration, the cops occupy a stable position.  If, on any robber turn, the cops {\em do not} occupy a stable position, then (as explained above) the robber moves to an undefended escape vertex, moves from there to the reset clique, and forever after remains safely in the reset clique.  Thus, so long as the robber remains on the robber track, the cops must always maintain a stable position (unless, of course, some cop can capture the robber with her next move).  Conversely, so long as the cops always occupy a stable position, they defend all $k$ escape vertices, so the robber cannot leave the robber track.  

In a stable position, we have one cop in $C_i$ for each $i$ -- that is, we have one cop in each cop track.  To maintain a stable position, each cop must remain within her track, and consequently must move ``forward'' on each turn.  That is, the cop who begins on $c_{i,0}$ must first move to $c_{i,1}$, then to $c_{i,2}$, and so forth.  Likewise, since the cops always occupy a stable position, the robber must move from $r_0$ to $r_1$, then to $r_2$, and so forth.  Once cop $i$ reaches $c_{i,q_i-1}$, there are two reasonable possibilities for her next move.  If the robber occupies $r_{p-1}^i$ at that time, then cop $i$ may (and surely will) capture him.  If instead the cops cannot capture the robber, then to maintain a stable position, cop $i$ must return to $c_{i,0}$.  Similarly, once the robber reaches $r_{p-2}$, he has some flexibility.  Assuming the cops occupy a stable position, the robber cannot leave the robber track.  Thus, if any vertex $r_{p-1}^i$ is undefended, the robber moves there, and the game continues as before -- with the robber moving next to $r_0$, then to $r_1$, and so forth.  If instead all of the $r_{p-1}^i$ are defended, then the robber moves to $r_{p-1}^0$, where he will be captured.  

This process can end only with the robber's capture.  This occurs if, and only if, on some robber turn, cop $i$ occupies vertex $c_{i,q_i-1}$ for all $i$, while the robber occupies $r_{p-2}$.  We refer to this as a {\em terminal configuration}.  Suppose that the game first reaches a terminal configuration in the $T$th round after reaching the initial configuration.  Since each cop must walk along her track over and over, never leaving and never pausing, we see that $T$ must be congruent to $-1$ modulo $q_i$ for all $i$.  Likewise, $T$ must be congruent to $-1$ modulo $p$.  We now choose $p, q_0, \dots, q_{k-1}$.  Fix an arbitrary positive integer $r$.  Let $p$ be the $r$th smallest prime number, $q_0$ the next-smallest, $q_1$ the next-smallest after that, and so forth.  Since $p$ and the $q_i$ are all prime, $T$ must be congruent to $-1$ modulo $pq_0q_1\dots q_{k-1}$, hence $T \ge pq_0q_1\dots q_{k-1}-1$.  

At this point, we note that $\cnumdir(G) \le k$.  If the cops all start on vertex $\omega$, then the robber must start in the reset clique.  The cops can now easily force the game into an initial configuration.  Henceforth, if the cops continue to follow the cop tracks, then the robber can never leave the robber track, so the game reaches a terminal configuration after $T$ rounds -- at which point the cops win.

Returning to the matter of $\captdir(G)$, recall that the $r$th-smallest prime number lies between $r(\log r + \log \log r - 1)$ and $r(\log r + \log \log r)$.  Thus,
$$\captdir(G) \ge T \ge [r(\log r + \log \log r - 1)]^{k+1}.$$
In addition,
\begin{align*}
n = \size{V(G)} &= \size{S} + \sum_{i=0}^{k-1}\size{C_i} + \size{R} + \size{X} + 1\\
                &\le k + k(r+k)(\log (r+k) + \log \log (r+k)) + r(\log r + \log \log r) + k + 1\\
								&\le 2kr(\log r + \log \log r - 1),
\end{align*}
for $r$ sufficiently large relative to $k$.  Thus,
$$\captdir(G) \ge [r(\log r + \log \log r - 1)]^{k+1} \ge \left (\frac{n}{2k}\right )^{k+1},$$
as claimed.
\end{proof}

\section{Directed Graphs}\label{sec:directed}

While our primary goal in this paper was to construct undirected graphs with large capture time, the tools we have established enable us to say a few things about directed graphs.  Most notably, Theorem~\ref{thm:directed} shows that the Cops and Robbers played on directed graphs is very closely connected to the game played on undirected graphs.  This is significant because Cops and Robbers on directed graphs is not well understood; very little work has been done in the area.  It is our hope that the techniques used in Theorem~\ref{thm:directed}, if not the theorem itself, can be used to establish new results on Cops and Robbers in the directed setting.

What work has been done on this topic -- most notably in \cite{FKL12}, \cite{GR95}, and \cite{LO17} -- has focused on strongly-connected directed graphs.  Theorem~\ref{thm:main} implies that for $k \ge 2$, there exist $n$-vertex strongly-connected directed graphs with cop number $k$ and capture time $\Theta(n^{k+1})$: simply construct an undirected graph with these properties and replace each undirected edge $uv$ with the two edges $\overrightarrow{uv}$ and $\overrightarrow{vu}$.  Moreover, the argument used to prove Proposition~\ref{prop:main_upper} can be applied to directed graphs, giving an $O(n^{n+1})$ bound on the capture time.

However, one thing is not immediately clear: how large can the capture time be for an $n$-vertex strongly-connected directed graph with cop number 1?  One cannot simply apply the argument that Bonato et al.~\cite{BGHK09} used to show that an undirected graph with cop number 1 has capture time $O(n)$; it does not extend to this more general setting.  There is good reason for this: in fact there exist $n$-vertex strongly-connected directed graphs with cop number 1 and capture time $\Omega(n^2)$, as we next show.

While Theorem~\ref{thm:main} requires $k \ge 2$, this is only needed to satisfy the hypotheses of Theorem~\ref{thm:directed}: when $k=1$, the main construction does in fact create a directed graph with cop number 1 and capture time $\Theta(n^2)$.  (Note that this shows that the bound on $k$ in Theorem~\ref{thm:main} is best possible: one cannot hope to adjust the construction so that it works when $k=1$.)  It is not hard to adjust this construction so that the graph produced is also reflexive and strongly-connected.  Our construction uses protected directed graphs; we remark that Lemma~\ref{lem:protected} can be extended to the directed graph setting by making the natural adjustments to the proof (\cite{Mam13}, Lemma 3.1).\footnote{
In particular:
\begin{itemize}
\item Take $P$ to be a doubly-directed incidence graph of a projective plane;
\item Add an edge from $(i,p)$ to $(j,q)$ if and only if $\overrightarrow{pq} \in E(P)$ or $i=j$;
\item $\overrightarrow{vw} \in E(G')$ when there is an unprotected edge from $\pi(v)$ to $\pi(w)$ in $G$ or when $\overrightarrow{vw} \in E(H)$ and either $\pi(v) = \pi(w)$ or there is a protected edge from $\pi(v)$ to $\pi(w)$ in $G$.
\end{itemize}}

\begin{theorem}\label{thm:directed_capture}
The maximum capture time among $n$-vertex strongly-connected directed graphs with cop number 1 is $\Theta(n^2)$.
\end{theorem}
\begin{proof}
As noted above, the argument used to prove Proposition~\ref{prop:main_upper} shows that every $n$-vertex directed graph with cop number 1 has capture time $O(n^2)$.  To establish the matching lower bound, we need only construct a strongly-connected protected directed graph $G$ with cop number 1 and capture time $\Omega(n^2)$; we may then apply Lemma~\ref{lem:protected} to obtain a corresponding protected undirected graph.

We use a construction similar to that used in Theorem~\ref{thm:main} (with $k=1)$, with only a few modifications.  (Since $k=1$, to simplify the notation, we write $c_i$, $r_{p-1}$, $q$, and $s$ in place of $c_{i,0}$, $r_{p-1}^0$, $q_0$, and $s_0$, respectively.)  First, we add protected loops at every vertex of $G$.  Next, we add a new vertex $\psi$, unprotected edges from every vertex on the cop vertex to $\psi$, an unprotected edge from $\psi$ to $\omega$, and an unprotected edge from $\omega$ to $\psi$.  (These new edges allow the cop to return to $\omega$, but he must take two steps to do so; this gives the robber time to escape back to the starter clique.)  We also add an unprotected edge from the escape vertex $s$ to $\omega$.  The graph is now strongly-connected.  From any vertex on the cop track, one can reach $\omega$ by way of $\psi$; from a vertex on the robber track, one can reach $\omega$ by way of $s$; from any vertex in the starter clique, one can reach $\omega$ by first entering the robber track.  From $\omega$, one can proceed directly to any vertex except the single vertex in the starter clique, which we can reach from $s$.  Thus for every vertices $u$ and $v$ in $G$ there is a path from $u$ to $\omega$ to $v$, so $G$ is strongly-connected.

Additionally, for each vertex $c_i$ on the cop track, we add a second vertex $c'_{i}$ with the same in-neighbors and out-neighbors as $c_{i}$.  Likewise, for each vertex $r_i$, we add a twin vertex $r'_i$.  Finally, for all $i \in \{0, \dots, q-1\}$ and all $j \in \{0, \dots p-1\}$ we add unprotected edges from $c_{i}$ to $r_j$ and from $c'_{i}$ to $r'_j$.  These edges will let the cop force the robber to move forward in his track, rather than following loops.  Note that by construction of the $c'_i$ and $r'_i$, there are also edges from $c_{q-1}$ to $r'_{p-1}$ and from $c'_{q-1}$ to $r_{p-1}$; as in Theorem~\ref{thm:main}, these edges will allow the cop to eventually capture the robber.

It is clear that $G$ is reflexive and strongly-connected.  To see that $G$ has cop number 1, we explain how one cop can capture the robber.  The cop begins on $\omega$.  To avoid immediate capture, the robber must begin in the starter clique.  The cop now moves to $c_0$, forcing the robber to leave the clique.  The robber must move to $r'_0$ (note that moving to $r_0$ would result in capture).  The cop next moves to $c'_1$.  The robber cannot remain where he is, nor can he move to the escape vertex $s$ or to $r'_1$; his only option is to move to $r_1$.  In response the cop moves to $c_2$, and so on.  As in the proof of Theorem~\ref{thm:main}, this process continues until the cop occupies $c_{q-1}$ (respectively, $c'_{q-1}$) while the robber occupies $r_{p-2}$ (resp. $r'_{p-2}$) on the robber's turn; the robber has no way to avoid capture on the cop's next turn.

Finally, we give a robber strategy showing that $G$ has capture time $\Omega(n^2)$.  The robber begins in the starter clique and remains there until the cop moves to either $c_0$ or $c'_0$.  The robber then moves to either $r'_0$ or $r_0$, respectively.  If the cop moves to $c_1$ then the robber moves to $r'_1$; if the cop moves to $c'_1$ then the robber moves to $r_1$; if the cop remains in place, then so does the robber; if the cop moves anywhere else, then the robber moves to $s$ and subsequently back to the starter clique.  The robber plays similarly on all future turns.  Under optimal play by the cop, the robber can never return to the starter clique, so we may suppose both players remain on their tracks.  Moreover, under optimal play the cop clearly never remains in place, since this simply wastes a turn.  Thus the cop and robber both keep moving forward along their tracks until the cop occupies $c_{q-1}$ (respectively, $c'_{q-1}$) while the robber occupies $r_{p-2}$ (resp. $r'_{p-2}$) on the robber's turn, after which the cop's win is ensured.  As in Theorem~\ref{thm:main}, choosing $p$ and $q$ to be sufficiently large consecutive primes now yields the desired conclusion.
\end{proof}

\section{Computational Complexity}\label{sec:complexity}

We close the paper by mentioning an interesting corollary of Theorem~\ref{thm:directed} in the area of computational complexity.  Let $\textsc{C\&R}(G,k)$ denote the decision problem associated with determining whether $k$ cops can capture a robber on the undirected graph $G$.  Goldstein and Reingold \cite{GR95} conjectured in 1995 that $\textsc{C\&R}$ is complete for the complexity class EXPTIME -- the class of decision problems solvable in time $O(2^{p(n)})$ for some polynomial $p$ (where, as usual, $n$ denotes the size of the input).  This conjecture was recently affirmed by the author using a rather involved argument (see~\cite{Kin15}).  As it turns out, Theorem \ref{thm:directed} yields a proof that is considerably shorter and, arguably, more elegant.

In addition to $\textsc{C\&R}$, we will need to refer to the following decision problems corresponding to variants of Cops and Robbers:
\begin{itemize}
\item $\textsc{C\&Rp}(G,k)$, wherein $G$ is a protected undirected graph;
\item $\textsc{C\&Rpd}(G,k)$, wherein $G$ is a protected directed (not necessarily reflexive) graph in which each vertex has at least one in-neighbor and one out-neighbor;
\item $\textsc{C\&Rdsc}(G,k)$, wherein $G$ is a strongly-connected directed reflexive graph.
\end{itemize}
While Goldstein and Reingold could not resolve the complexity of $\textsc{C\&R}$, they did show that $\textsc{C\&Rdsc}$ is EXPTIME-complete (\cite{GR95}, Theorem 4).  This result, in conjunction with Theorem~\ref{thm:directed} and Mamino's result that $\textsc{C\&Rp}$ reduces to $\textsc{C\&R}$, yields a short proof that $\textsc{C\&R}$ is EXPTIME-complete.

\begin{cor}
$\textsc{C\&R}$ is EXPTIME-complete.
\end{cor}
\begin{proof}
$\textsc{C\&R}$ is easily seen to belong to EXPTIME, so it suffices to show that it is EXPTIME-hard.  $\textsc{C\&Rdsc}$ trivially reduces to $\textsc{C\&Rpd}$, since an unprotected directed graph can be viewed as a protected directed graph in which each edge happens to be unprotected.  Theorem~\ref{thm:directed} shows that $\textsc{C\&Rpd}$ reduces to $\textsc{C\&Rp}$.  Finally, $\textsc{C\&Rp}$ reduces to $\textsc{C\&R}$ (\cite{Mam13}, Lemma 3.1).  Since $\textsc{C\&Rdsc}$ is EXPTIME-hard (\cite{GR95}, Theorem 4), it follows that $\textsc{C\&R}$ is also EXPTIME-hard.
\end{proof}

\end{section}

\begin{thebibliography}{99}

\bibitem{AF84} M. Aigner, M. Fromme, A game of cops and robbers, {\em Discrete Applied Mathematics} 8 (1984), 1--12.

\bibitem{BEUW17} S. Brandt, Y. Emek, J. Uitto, and R. Wattenhoffer, A tight lower bound for the capture time of the Cops and Robbers game, {\em Proceedings of 44th International Colloquium on Automata, Languages, and Programming}, 2017, to appear.  [Available online at https://ie.technion.ac.il/~yemek/Publications/tlbctcrg.pdf.]
 
\bibitem{BGHK09} A. Bonato, P. Golovach, G. Hahn, and J. Kratochv\'il, The capture time of a graph, {\em Discrete Mathematics} 309 (2009), 5588--5595.

\bibitem{BGKP13} A. Bonato, P. Gordinowicz, W.B. Kinnersley, and P. Pra\l{}at, The capture time of the hypercube, {\it Electon. J. Combin.} 20 (2013), Paper P24.

\bibitem{BI93} A. Berarducci and B. Intrigila, On the cop number of a graph, {\em Advances in Applied Mathematics} 14 (1993), 389--403.

\bibitem{BN11} A. Bonato and R.J. Nowakowski, {\em The Game of Cops and Robbers on Graphs}, American Mathematical Society, Providence, Rhode Island, 2011.
 
\bibitem{BPPR17} A. Bonato, X. P\'{e}rez-Gim\'{e}nez, P. Pra\l{}at, and B. Reiniger, The game of Overprescribed Cops and Robbers played on graphs, {\em Graphs and Combinatorics} (2017). doi:10.1007/s00373-017-1815-2.

\bibitem{FKL12} A. Frieze, M. Krivilevich, and P. Loh, Variations on cops and robbers, {\em Journal of Graph Theory} 69 (2012), no. 4, 383--402.

\bibitem{FNUW15} K.-T. F\"{o}rster, R. Nuridini, J. Uitto, and R. Wattenhoffer, Lower bounds for the capture time: Linear, quadratic, and beyond, {\em Lecture Notes in Computer Science} 9439 (2015), 342--356.

\bibitem{Gav10} T. Gaven\v{c}iak, Cop-win graphs with maximum capture-time, {\em Discrete Mathematics} 310 (2010), 1557--1563.

\bibitem{GR95} A.S. Goldstein and E.M. Reingold, The complexity of pursuit on a graph, {\em Theoretical Computer Science} 143 (1995), 93--112.

\bibitem{Kin15} W. B. Kinnersley, Cops and Robbers is EXPTIME-complete, {\em Journal of Combinatorial Theory Series B} 111 (2015), 201--220.

\bibitem{LO17} P. Loh and S. Oh, Cops and Robbers on Planar-Directed Graphs, {\em Journal of Graph Theory} 86 (2017), no. 3, 329--340.

\bibitem{Mam13} M. Mamino, On the computational complexity of a game of cops and robbers, {\em Theoretical Computer Science} 477 (2013), 48--56.

\bibitem{Meh11} A. Mehrabian, The capture time of grids, {\em Discrete Mathematics} 311 (2011), 102--105.

\bibitem{NW83} R.J. Nowakowski and P. Winkler, Vertex-to-vertex pursuit in a graph, {\em Discrete Mathematics} 43 (1983), 235--239.

\bibitem{Pis16} P. Pisantechakool, On the Capture Time of Cops and Robbers Game on a Planar Graph. In: Chan TH., Li M., Wang L. (eds) {\em Combinatorial Optimization and Applications. COCOA 2016. Lecture Notes in Computer Science}, 10043.

\bibitem{Qui78} A. Quilliot, Jeux et pointes fixes sur les graphes, Th\`{e}se de 3\`{e}me cycle, Univerit\'{e} de Paris VI, 1978, 131--145.

\bibitem{Wes01} D.B. West, {\em Introduction to Graph Theory, 2nd edition}, Prentice Hall, 2001.

\end{thebibliography}
\end{document}